\documentclass{amsproc}
\usepackage{amsfonts}
\usepackage{amssymb}
\usepackage{amsmath}
\usepackage{tikz}
\usepackage{graphicx}
\newtheorem{theorem}{\textbf{Theorem}}

\newtheorem{definition}{\textbf{Definition}}

\newtheorem{lemma}{\textbf{Lemma}}
\newtheorem{proposition}{\textbf{Proposition}}
\newtheorem{remark}{\textbf{Remark}}

\begin{document}

\title[On $\mathbb{P}$-Weakly Hyperbolic Iterated Function Systems]{ On $\mathbb{P}$-Weakly Hyperbolic Iterated Function Systems}

\author[\'Italo Melo]{\'Italo Melo}


\begin{abstract}
In this paper we will consider the
concept of $\mathbb{P}$-weakly hyperbolic iterated function systems on compact metric spaces that generalizes the concept of weakly hyperbolic iterated function systems, as defined by Edalat in \cite{E} and by Arbieto, Santiago and Junqueira in \cite{ASJ} for a more general setting where the parameter space is a compact metric space. We prove the existence and uniqueness of the invariant measure of a $\mathbb{P}$-weakly hyperbolic IFS. Furthermore, we prove  an ergodic theorem  for  $\mathbb{P}$-weakly hyperbolic IFS with compact parameter space.


\medskip

\noindent\textit{$\bf{Key \ words:}$} Iterated function systems, invariant measure, ergodic theorem. 

\vspace{.2cm}


\medskip

\begin{center}Universidade Federal do Piau\'{i}, Departamento de
Matem\'{a}tica, 64049-550, Ininga - Teresina - PI, Brasil\\
e-mail: italodowell@ufpi.edu.br\\
\end{center}

\end{abstract}

\maketitle

\section{Introduction}

Iterated function systems (IFS) were introduced in \cite{HU} although some results appeared earlier in \cite{W}. The theory of IFS has found applications in diverse \linebreak scientific areas, like as image compression  \cite{B1}.

In the famous paper \cite{HU}, Hutchinson  considers hyperbolic Iterated function systems, i.e, a finite collection of contractions $\varphi_1,...,\varphi_N: X \rightarrow X$, where $X$ is a complete metric space. He proved that there is a only compact set $K$, called the invariant set of the IFS, such that
$$K = \displaystyle\bigcup_{i=1}^n \varphi_i(K).$$

He also proved that there exists a unique probability Borel measure $\mu$, called the invariant measure, satisfying the equation
$$\mu = \displaystyle\sum_{i=1}^N p_i {\varphi_i}_{*}\mu,$$
where $*$ is the push-forward operator and  $p_1,...,p_N$ are positives numbers such that $p_1+ \cdot \cdot \cdot + p_N= 1$. Furthermore, ${\rm supp}\, (\mu) = K$. 
After these results, many authors generalized  the Hutchinson's results assuming a weak form of contraction, see for example \cite{Dem}, \cite{ET} and \cite{MA} .


In \cite{E}, Edalat introduced the notion of a weakly hyperbolic IFS on a compact metric space $X$ that generalizes the notion of a hyperbolic IFS. In this context he generalized the Hutchinson's results.



Another natural generalization of a hyperbolic IFS with finite maps is related with the parameter space. Hata in \cite{HATA} considered a parameter space infinite countable where the maps are weak contractions and 
Mendivil in \cite{Me} studied the \linebreak invariant measures for a IFS with compact parameter space where the maps are average contractions. 	


 Arbieto et al. in \cite{ASJ}, study weakly IFS on compact metric spaces, but in the more general setting of a compact parameter space, thus
unifying and extending some of the previous results. In particular, they generalize the Edalat's results.


 They proved the existence of attractors, both in the topological and measure theoretical viewpoint and defined weakly hyperbolic iterated function systems for complete spaces and compact parameter space. Furthermore, they studied the question of existence of the attractors in this setting and also proved a version of the results \cite{BV}, about drawing the attractor (also called the chaos game), for the case of compact parameter space. 


In \cite{DM}, D\'iaz and Matias study iterated function systems with finite maps that are not weakly hyperbolic but that has some hyperbolic condition and they provides a necessary and sufficient condition for the existence of a globally attracting fixed point of the Barnsley-Hutchinson operator. Furthermore, for a recurrent IFS defined on $[0,1]$ they study the asymptotic stability of the Markov operator associated to this IFS.
  
  
  Elton proved in \cite{ELTON} an ergodic theorem for a IFS contractive on average with parameter space finite. This theorem was extended by Barnsley, Elton and Hardin in \cite{B3} to recurrent IFS and by Stenflo in \cite{St} for a general IFS with time-dependent probabilities. Cong and Siegmund in \cite{ET} proved an ergodic theorem for iterated function systems on compact metric spaces with countably
  many functions where the maps are contractive on average.
  
  Recently in \cite{ASJ} the authors proved an ergodic theorem for a weakly hyperbolic IFS with compact parameter space.

  In this paper we will consider the
  concept of $\mathbb{P}$-weakly hyperbolic iterated function systems on compact metric spaces with compact parameter space that generalizes the concept of weakly hyperbolic iterated function systems. We prove the existence and uniqueness of the invariant measure of a $\mathbb{P}$-weakly hyperbolic.\linebreak Furthermore, we get an ergodic theorem  for  $\mathbb{P}$-weakly hyperbolic IFS with compact parameter space.

\section{Notation and Preliminaries}
In this section we present the notations and some preliminaries results that will be used in the sequel.

\subsection{Iterated function system}\label{OP1}
An iterated function system (IFS) on a \linebreak metric space $X$ is given by a finite set of continuous map $f_i:X \rightarrow X $ with \linebreak $i=1,...,N$. An IFS is hyperbolic if $X$ is a complete metric space and all the maps are contractions.

We denote by $\mathcal{K}(X)$ the class of all compact non-empty subsets of $X$ and by $d_H$ the Hausdorff metric,  it is know that  $\mathcal{K}(X)$ is a complete metric space with respect to the metric $d_H$ if $X$ is a complete metric space.

A hyperbolic IFS induces the operator $\mathcal{F}: \mathcal{K}(X) \rightarrow \mathcal{K}(X)$ defined by
$$\mathcal{F}(A) = \displaystyle\bigcup_{i=1}^N  f_i(A).$$

This operator is know like Barnsley-Hutchinson operator. Hutchinson proved in \cite{HU} that $\mathcal{F}$ is a contraction thus by Banach contraction principle the operator $\mathcal{F}$ has a unique fixed point $K$, called the invariant set of the IFS. Furthermore, we have

$$ K = \displaystyle\lim_{n \rightarrow \infty} \mathcal{F}^{n}(A),$$ 
in the Hausdorff metric for any $A \in \mathcal{K}(X).$


\subsection{The invariant measure of an IFS with probabilities}\label{OP3}

An iterated function system with probabilities (IFSp) on metric space $(X,d)$ is given by a finite set of continuous map $f_i:X \rightarrow X $ with $i=1,...,N$ and a vector of probability \linebreak
$p=(p_1,...,p_N)$ with $p_i > 0$. If $X$ is a complete and separable metric space, Kravchenko proved in \cite{K} that $(\mathcal{M}(X), H)$ is a complete metric space, where  $\mathcal{M}(X)$ is the set of the probability Borel measures $\mu$ such that $\displaystyle\int_X f d\mu < + \infty$ for each $f \in Lip_1(X, \mathbb{R})$,  where
$$   Lip_1(X, \mathbb{R}) = \{f:X \rightarrow \mathbb{R}: |f(x)-f(y)| \leq d(x,y) \ \ {\rm{for}} \ \ {\rm{all}}  \ \ x,y \in X\}.$$

The Hutchinson metric $L$ can be defined  on $\mathcal{M}(X)$ as follows

$$   L(\mu,\nu) = \sup \Big\{\Big| \displaystyle\int_X f d\mu - \displaystyle\int_X f d\nu \Big|  : f \in  Lip_1(X, \mathbb{R})    \Big\}.                 $$
Hutchinson proved that if the IFS is hyperbolic then for each vector of probability $p$ the Markov operator $T_p: \mathcal{M}(X) \rightarrow \mathcal{M}(X)$ given by

$$T_p(\mu) = \displaystyle\sum_{i=1}^{N}p_i {f_i}_{*}(\mu),  $$is a contraction with respect to Hutchinson metric $L$ thus $T_p$ has a unique fixed point $\mu_p$, called the invariant measure for the IFS or the stationary measure.


Let $\mathcal{M}_1(X)$ be the set of the Borel probability measures on $X$ equipped with the wea$k^{*}$ topology. If $X$ is a compact metric space then the sets $\mathcal{M}_1(X)$ and 
$\mathcal{M}(X)$ are equal. Kravchenko proved in \cite{K} that the Hutchinson topology and the wea$k^{*}$ topology are equivalent in this case.



\subsection{Weakly hyperbolic IFS}

Following the notation of \cite{ASJ}. Let $\Lambda$ and $X$ be compact metric spaces. A continuous map $\omega:\Lambda \times X \rightarrow X$ is called an Iterated Function System (IFS). The space $\Lambda$ is called the parameter space and $X$ is called the phase space. The space $\Lambda^{\mathbb{N}}$ of infinite words with alphabet in $\Lambda$, endowed with the product topology will be denoted by $\Omega$.

Given a fixed parameter $\lambda \in \Lambda$, we will denote by $\omega_{\lambda }:X \rightarrow X$ the map defined by $\omega_{\lambda}(x) = \omega(\lambda,x)$.
\begin{definition}\label{weakly}
	An IFS $\omega$ on the metric space $(X,d)$ is weakly hyperbolic if for every $\sigma \in \Omega$ we have
	$$   \displaystyle\lim_{n \rightarrow \infty} Diam(f_{\sigma_1} \circ \cdot \cdot \cdot \circ f_{\sigma_n}(X) ) = 0,         $$
	where $Diam(A) = \sup\{d(x,y): x,y \in A\}$.
\end{definition}
If $\Lambda$ is finite then this definition coincides with the definition introduced by Edalat in \cite{E}. A IFS $\omega$ induces the operator $\mathcal{F}: \mathcal{K}(X) \rightarrow \mathcal{K}(X)$ defined by
$$\mathcal{F}(A) = \displaystyle\bigcup_{\lambda \in \Lambda} \omega_{\lambda}(A).$$

 Given a probability $p$ in $\Lambda$ we denote by $\mathbb{P}$ the product measure in $\Omega$ induced by $p$, recall that the product measure $\mathbb{P}$ in $\Omega$ is an invariant measure by the shift map $\beta: \Omega \rightarrow \Omega$ defined by 
 $$   \beta(\sigma_1,\sigma_2,...) = (\sigma_2,\sigma_3,...).$$
 



 In \cite{ASJ} the authors consider the operator $T_{p}: \mathcal{M}_1(X) \rightarrow \mathcal{M}_1(X)$, defined by
$$    T_{p}(\mu)(B) =  \displaystyle\int_{\Lambda} \mu(\omega_{\lambda}^{-1}(B)) dp(\lambda),                $$
for every Borel set $B$. We say that $\mu$ is an invariant measure for $\omega$ if $\mu$ is a fixed point of the transfer operator. Observe that if $\Lambda$ is finite then this operator coincides with the Markov operator defined by Hutchinson.

In the context of weakly hyperbolic IFS on  compact metric spaces $X$ with parameter space finite Edalat in \cite{E} generalizes the Hutchinson's results. Arbieto et al. in \cite{ASJ} consider weakly hyperbolic IFS with compact parameter space and generalize the Edalat's results.

In the  Theorem 1 of \cite{ASJ} they prove the existence of global attractors for weakly hyperbolic IFS with compact parameter space. Let $p$ a probability in $\Lambda$, they also proved that each transfer operator $T_p$ has a unique fixed point $\mu_p$. Furthermore the operator $T_p$ is asymptotically stable, i.e,  $T_{p}^n (\nu) \rightarrow \mu_p$ in the $weak^{*}$ topology for every measure $\nu \in \mathcal{M}_1(X)$.

\section{$\mathbb{P}$-Weakly Hyperbolic Iterated Function Systems}
 The purpose of this subsection is to introduce the concept of $\mathbb{P}$-weakly \linebreak hyperbolic iterated function systems. Let $(X,d),(\Lambda, \rho)$ be compact metric spaces and $\omega:\Lambda \times X \rightarrow X$ an IFS.
\begin{definition}
	Fix $p \in \mathcal{M}_1(\Lambda)$. We say that the IFS $\omega$ is $\mathbb{P}$-weakly hyperbolic if $\mathbb{P}(S) = 1$, where
	$$       S = \{\sigma\in \Omega: \displaystyle\lim_{n \rightarrow \infty} Diam(\omega_{\sigma_1} \circ \cdot \cdot \cdot \circ \omega_{\sigma_n}(X) ) = 0 \} .       $$              
\end{definition} 
  Consider the functions $f_n,h_n:\Omega \rightarrow \mathbb{R}$, where 
 $$ f_n(\sigma) = Diam(  \omega_{\sigma_{n}} \circ \cdot \cdot \cdot \circ \omega_{\sigma_2} \circ\omega_{\sigma_1}(X))      $$ 
 and 
 $$h_n(\sigma) = Diam(\omega_{\sigma_1} \circ \cdot \cdot \cdot \circ \omega_{\sigma_n}(X)).$$
 
 If the IFS $\omega$ is weakly hyperbolic by Lemma 2.2 of \cite{ASJ} we have that the functions $f_n$ and $h_n$ converge uniformly to $0$. If the IFS is only $\mathbb{P}$- weakly hyperbolic we lose the uniform convergence to zero. Furthermore, in some cases there exists $\sigma \in \Omega$ such that $\displaystyle\lim_{n\to+\infty} h_n(\sigma) > 0$ or $\displaystyle\lim_{n\to+\infty} f_n(\sigma) $ does not exist.

 
  We denote by $D$ and $r$ the metrics of $\Omega$ and $\Omega \times X$, respectively. The metric $D$ is given by
 $$          D(\sigma,\xi) = \displaystyle\sup_{n \in \mathbb{N}} \displaystyle\frac{1}{n} \rho(\sigma_n,\xi_n).                                   $$

D\'iaz and Matias also study in \cite{DM} properties of the set $S$ and the asymptotic stability of the Markov operator $T_p$ for a IFS $\omega$ whose parameter space is finite. The notion of $\mathbb{P}$-weakly hyperbolic IFS generalizes the notion of weakly hyperbolic IFS since clearly a weakly hyperbolic IFS is $\mathbb{P}$-weakly hyperbolic.

If the IFS $\omega$ is weakly hyperbolic from Theorem 2 of \cite{ASJ} it follows that for each $p \in \mathcal{M}_1(\Lambda)$ the operator $T_p$ has a unique fixed point. Furthermore, the operator $T_p$ is asymptotic stable. Our first result generalizes this theorem for $\mathbb{P}$-weakly hyperbolic IFS with compact parameter space.

\begin{theorem}\label{teoo}
Fix $p \in \mathcal{M}_1(\Lambda)$. If $X$ is a compact metric space and $\omega$ is a \linebreak $\mathbb{P}$-weakly hyperbolic IFS with compact parameter space then the operator \linebreak $T_p: \mathcal{M}_1(X) \rightarrow \mathcal{M}_1(X)$ has a unique fixed point $\mu_p$ and $T_{p}^n (\nu) \rightarrow \mu_p$ in the $weak^{*}$ topology for every, $\nu \in \mathcal{M}_1(X)$. Furthermore, if $p(U) > 0$ for every open set $U \subset \Lambda$ then ${\rm supp}\,(\mu) = K$, where $K = \overline{\Gamma(S)}$.
\end{theorem}


In the next result we get an ergodic theorem for $\mathbb{P}$-weakly hyperbolic IFS that generalizes the Theorem 3 of \cite{ASJ}.


 
  \begin{theorem}\label{ergodic}(Ergodic Theorem for $\mathbb{P}$-weakly hyperbolic IFS)
 Fix $p \in \mathcal{M}_1(\Lambda)$. If the IFS $\omega$ is $\mathbb{P}$-Weakly Hyperbolic then for any continuous function $f:X \rightarrow \mathbb{R}$, any $x \in X$ and $\mathbb{P}$-almost every $\sigma \in \Omega$ we have:
 	
 	$$  \displaystyle\lim_{n\to+\infty}    \displaystyle\frac{1}{n} \displaystyle\sum_{j=1}^{n} f(  \omega_{\sigma_j} \circ \cdot \cdot \cdot \circ \omega_{\sigma_1}(x)   )     = \displaystyle\int_{X} f d\mu_p,                                                                   $$ 
 	where $\mu_p$ is the unique invariant measure associated to the operator $T_p$.
 \end{theorem}
\section{Proof of Theorem 1}
Let $\varphi$ be a continuous function on $X$ and $\nu$ a Borel probability measure on $X$. Given $\epsilon > 0$, by uniform continuity there exists $\delta > 0$ such that if $d(x,y) < \delta$ then $|\varphi(x)-\varphi(y)| < \epsilon/2$. Since $\mathbb{P}(S) = 1$, by Egorov's theorem there exists a Borel set $B \subset \Omega$ with  $\mathbb{P}(B) > 1-\epsilon/4M$ such that $h_n$ converges to $0$ uniformly on $B$, where $M = \sup |\varphi(x)| $.

By definition of $B$ there exists $n_0 \in \mathbb{N}$ such that if $n \geq n_0$ then $h_{n}(\sigma) < \delta$ for any $\sigma \in B$. Take $\sigma \in B$, for any $x \in X$ and $m,n \geq n_0$ we have
$$      d(      \omega_{\sigma_1} \circ \cdot \cdot \cdot \circ \omega_{\sigma_n}(x)        ,  \omega_{\sigma_1} \circ \cdot \cdot \cdot \circ \omega_{\sigma_m}(x)    ) < \delta.                                     $$
Hence, $|\varphi(\omega_{\sigma_1} \circ \cdot \cdot \cdot \circ \omega_{\sigma_n}(x)  ) - \varphi( \omega_{\sigma_1} \circ \cdot \cdot \cdot \circ \omega_{\sigma_m}(x)   )| < \epsilon/2$. On the other hand, from definition of the transfer operator $T_p$ follows that
$$         \displaystyle\int_{X} \varphi d(T_p^n(\nu)) =     \displaystyle\int_{\Lambda^n}    \displaystyle\int_{X}  \varphi (\omega_{\sigma_1} \circ \cdot \cdot \cdot \circ \omega_{\sigma_n}(x)) d\nu dp^n.$$
Observe that 
we can split
any integration in $\Omega$ as an integration in $\Lambda^n \times \Omega$. Using this, we get
$$  \displaystyle\int_{\Lambda^n}    \displaystyle\int_{X}  \varphi (\omega_{\sigma_1} \circ \cdot \cdot \cdot \circ \omega_{\sigma_n}(x)) d\nu dp^n =  \displaystyle\int_{\Omega}    \displaystyle\int_{X}  \varphi (\omega_{\sigma_1} \circ \cdot \cdot \cdot \circ \omega_{\sigma_n}(x)) d\nu d\mathbb{P}.$$
On the other hand, for $m,n \geq n_0$ we have

$$	\Big|       \displaystyle\int_{B}\int_{X} (\varphi( \omega_{\sigma_1} \circ \cdot \cdot \cdot \circ \omega_{\sigma_n}(x)  )  - \varphi( \omega_{\sigma_1} \circ \cdot \cdot \cdot \circ \omega_{\sigma_m}(x)  ))  d\nu d\mathbb{P}                              \Big| \leq \frac{\epsilon}{2} \cdot \nu(B) $$
and
$$   \Big|       \displaystyle\int_{B^c}\int_{X} (\varphi( \omega_{\sigma_1} \circ \cdot \cdot \cdot \circ \omega_{\sigma_n}(x)  ) - \varphi( \omega_{\sigma_1} \circ \cdot \cdot \cdot \circ \omega_{\sigma_m}(x)  ))  d\nu d\mathbb{P}                        \Big| \leq  2\sup |\varphi(x)| \cdot \nu(B^c ).$$
Hence, for $m,n \geq n_0$ we get
$$ \Big| \displaystyle\int_{X} \varphi d(T_p^n(\nu))  -   \displaystyle\int_{X} \varphi d(T_p^m(\nu))                      \Big|  < \epsilon.$$
Therefore, 
$ \displaystyle\lim_{n\to+\infty} \displaystyle\int_{X} \varphi d(T_p^n(\nu)) $ exists for every continuous function $\varphi$. By Riesz-Markov representation theorem there exists a unique Borel measure $\mu$ such that 
$$ \displaystyle\lim_{n\to+\infty} \displaystyle\int_{X} \varphi d(T_p^n(\nu)) = \displaystyle\int_X \varphi d\mu,$$ 
for any continuous function $\varphi$ thus $ \displaystyle\lim_{n\to+\infty} T_{p}^n (\nu) = \mu$ in the wea$k^{*}$ topology. From Lemma 3.1 of \cite{ASJ} it follows that the Markov operator $T_p$ is continuous in the wea$k^{*}$ topology so $\mu$ is a fixed point of $T_p$.

We prove that for every $\nu \in \mathcal{M}_1(X)$ the sequence $\{T_{p}^n (\nu)\}_{n=1}^{\infty}$ converges in the wea$k^{*}$ topology to a fixed point of the operator $T_p$. To prove that the operator $T_p$ is asymptotic stable it suffices to prove that $T_p$ has a unique fixed point.

Now, we will prove a result that relates the set $S$ with the following set
$$   F = \{ \sigma \in \Omega:  \displaystyle\lim_{n\to+\infty}  \displaystyle\frac{1}{n} \displaystyle\sum_{i=1}^n Diam(\omega_{\sigma_i} \circ \cdot \cdot \cdot \circ \omega_{\sigma_1}(X))  = 0    \}                                .$$

 By Lemma 2.1 of \cite{ASJ} the functions $\psi_n:\Lambda^n \rightarrow \mathbb{R}$ given by 
 $$\psi_n(\sigma_1,...,\sigma_n) =   Diam(\omega_{\sigma_1} \circ \cdot \cdot \cdot \circ \omega_{\sigma_n}(X))$$
  are continuous and so the functions $f_n$ and $h_n$ are also continuous. Observe that
\[
\begin{split}
	h_{n+1}(\sigma)  &=     Diam(\omega_{\sigma_1} \circ \cdot \cdot \cdot \circ \omega_{\sigma_{n+1}}(X))  \\    
	& =      Diam(\omega_{\sigma_1} \circ \cdot \cdot \cdot \circ \omega_{\sigma_{n}}(\omega_{\sigma_{n+1}}(X))    \\
	&\leq h_n(\sigma).            \\    
\end{split}
\]
Thus $h(\sigma) = \displaystyle\lim_{n\to+\infty} h_n(\sigma)$ exists for every $\sigma \in \Omega$, clearly $h$ is measurable. We also have that
\[
\begin{split}
	f_{n+1}(\sigma) &= Diam(\omega_{\sigma_{n+1}} \circ \cdot \cdot \cdot \circ \omega_{\sigma_1}(X))\\
	&=Diam(\omega_{\sigma_{n+1}} \circ \cdot \cdot \cdot \circ \omega_{\sigma_2}(\omega_{\sigma_1}(X)))\\
	&\leq f_n(\beta(\sigma)).
\end{split}
\]
Hence, for $m,k\in \mathbb{N}$ we have that $f_{m+k}(\sigma) \leq f_{m}(\beta^k(\sigma))$. Define $u_n = \displaystyle\sum_{j=1}^n f_j$, observe that
$$
u_{m+n} =\displaystyle\sum_{j=1}^{m+n} f_j
= u_m + \displaystyle\sum_{j=m+1}^{m+n} f_j
\leq u_m + \displaystyle\sum_{i=1}^{n} f_i \circ \beta^m= u_m + u_n \circ \beta^m.
$$
Since the sequence $(u_n)_n$ is subadditive by Kingman's subadditive ergodic theorem the sequence $(u_n/n)_n$ converges $\mathbb{P}$-almost everywhere for an invariant function $f$ such that
$$ \displaystyle\int_{\Omega} f d\mathbb{P} = \displaystyle\lim_{n\to+\infty}  \displaystyle\frac{1}{n} \displaystyle\int_{\Omega} u_n d\mathbb{P}= \displaystyle\inf_{n}   \displaystyle\frac{1}{n} \displaystyle\int_{\Omega} u_n d\mathbb{P}                                  . $$

\begin{lemma}\label{lemma1}
	Fix $p \in \mathcal{M}_1(\Lambda)$. Then for every n, 
	$$\displaystyle\int_{\Omega} f_n d\mathbb{P} = \displaystyle\int_{\Omega} h_n d\mathbb{P}.$$
\end{lemma}
\begin{proof}
	Given $\epsilon > 0$, by uniform continuity there exists $\delta > 0$ such that if $D(\sigma,\xi) < \delta$ then $|f_n(\sigma) - f_n(\xi)| < \epsilon$ and $|h_n(\sigma) - h_n(\xi)| < \epsilon$. Consider a partition $\mathcal{P} = \{ P_1,...,P_k\}$ of $\Lambda$ with $Diam(P_i) < \delta/2$ and fix $\lambda_i \in P_i$	for each $i=1,...,k$.

	Fix $m > n$ such that $Diam(X)/m < \delta/2$ and $\sigma^{i_1,...,i_m} \in \Omega$ such that \linebreak $\sigma_{j}^{i_1,...,i_m} = \lambda_{i_j}$ for $j=1,...,m$ and $i_1,...,i_m =1,...,k$. Observe that $$  Diam( [1;P_{i_1},...P_{i_m}]     ) < \delta.$$
	Hence,
\[
	\begin{split}
		\displaystyle\int_{\Omega} f_n d\mathbb{P} &= \displaystyle\sum_{i_1,...,i_m = 1}^k \displaystyle\int_{  [1;P_{i_1},...P_{i_m}]      } f_n(\sigma) d\mathbb{P} \\
		& < \displaystyle\sum_{i_1,...,i_m = 1}^k  \displaystyle\int_{  [1;P_{i_1},...P_{i_m}]      } (\epsilon + f_n (\sigma^{i_1,...,i_m}  )) d\mathbb{P} \\
		&= \epsilon + \displaystyle\sum_{i_1,...,i_m = 1}^k  p(P_{i_1}) \cdot \cdot \cdot p(P_{i_m})  Diam(\omega_{ \lambda_{i_n}} \circ \cdot \cdot \cdot \circ \omega_{ \lambda_{i_1}}(X)) \\
		&=  \epsilon + \displaystyle\sum_{i_1,...,i_m = 1}^k  p(P_{i_n}) \cdot \cdot \cdot p(P_{i_1}) \cdot p(P_{i_{n+1}})\cdot \cdot  \cdot  p(P_{i_m})  h_n( \sigma^{i_n,...,i_1,i_{n+1},...,i_m}) \\		
		&= \epsilon + \displaystyle\sum_{i_1,...,i_m = 1}^k \displaystyle\int_{[1;P_{i_n},...,P_{i_1},P_{i_{n+1}},...,P_{i_m}]      }  h_n( \sigma^{i_n,...,i_1,i_{n+1},...,i_m}) d\mathbb{P}\\
		&< 2\epsilon + \displaystyle\sum_{i_1,...,i_m = 1}^k \displaystyle\int_{  [1;P_{i_n},...,P_{i_1},P_{i_{n+1}},...,P_{i_m}]      } h_n(\sigma) d\mathbb{P}\\
		&= 2\epsilon +  \displaystyle\int_{\Omega} h_n d\mathbb{P}.\\ 
	\end{split}
\]
	We also have 
	$$    \displaystyle\int_{\Omega} h_n d\mathbb{P} < 2\epsilon + \displaystyle\int_{\Omega} f_n d\mathbb{P}.                             $$
	Therefore, 
	$$\Big| \displaystyle\int_{\Omega} f_n d\mathbb{P} - \displaystyle\int_{\Omega} h_n d\mathbb{P}\Big| < 2\epsilon.$$
	For every $\epsilon > 0$, thus $\displaystyle\int_{\Omega} f_n d\mathbb{P} = \displaystyle\int_{\Omega} h_n d\mathbb{P} $.
\end{proof}

The following result is the key lemma of this section. We shall prove that $\mathbb{P}(F) = 1$ if and only if $\mathbb{P}(S) = 1$. In particular, if the IFS $\omega$ is $\mathbb{P}$-weakly hyperbolic then for almost every $\sigma \in \Omega$ we have
$$\displaystyle\lim_{n\to+\infty}  \displaystyle\frac{1}{n} \displaystyle\sum_{i=1}^n Diam(\omega_{\sigma_i} \circ \cdot \cdot \cdot \circ \omega_{\sigma_1}(X))  = 0.$$


\begin{lemma}\label{lemakey}
	Fix $p \in \mathcal{M}_1(\Lambda)$. Then $\mathbb{P}(F) = 1$ if and only if $\mathbb{P}(S) = 1$.
\end{lemma}
\begin{proof}
	Assume that $\mathbb{P}(F) = 1$, by Fatou's lemma we have
	\begin{eqnarray*}
		\displaystyle\int_{\Omega} h d\mathbb{P} &=& \displaystyle\int_{\Omega} \liminf_{n\to+\infty} h_n d\mathbb{P}\\
		&\leq& \liminf_{n\to+\infty} \displaystyle\int_{\Omega} h_n d\mathbb{P}\\
		& \leq& \liminf_{n\to+\infty} \displaystyle\frac{1}{n} \displaystyle\sum_{i=1}^n \displaystyle\int_{\Omega} h_i d\mathbb{P}.\\
	\end{eqnarray*}      
	From Lemma \ref{lemma1} it follows that
	\begin{eqnarray*}  
		\displaystyle\int_{\Omega} h d\mathbb{P} &\leq&           \liminf_{n\to+\infty} \displaystyle\frac{1}{n} \displaystyle\sum_{i=1}^n \displaystyle\int_{\Omega} f_i d\mathbb{P} \\  
		&=&  \liminf_{n\to+\infty} \displaystyle\frac{1}{n} \displaystyle\int_{\Omega} u_n d\mathbb{P}.
		 \\                   
	\end{eqnarray*}  
	On the other hand, $u_n(\sigma)/n$ converges $\mathbb{P}$-almost everywhere for $0$ . From dominated convergence theorem it follows that 
	$$ \displaystyle\int_{\Omega} h d\mathbb{P} = 0.$$
Since $h \geq 0$ almost everywhere it follows that $h = 0$ almost everywhere. Therefore, $ \mathbb{P}(S) = 1$.

	 Now assume that  $\mathbb{P}(S) = 1$, by definition $h_n$ converges $\mathbb{P}$-almost everywhere for $0$. From dominated convergence theorem it follows that
	$$      \displaystyle\lim_{n\to+\infty}    \displaystyle\int_{\Omega} h_n d\mathbb{P}      = 0.                 $$ 
	On the other hand,
	\begin{eqnarray*}
		\displaystyle\int_{\Omega} f d\mathbb{P} &=& \displaystyle\lim_{n\to+\infty}  \displaystyle\frac{1}{n} \displaystyle\int_{\Omega} u_n d\mathbb{P}\\
		&=& \displaystyle\lim_{n\to+\infty}\displaystyle\frac{1}{n} \displaystyle\sum_{j=1}^n \displaystyle\int_{\Omega} f_j d\mathbb{P}.\\
	\end{eqnarray*}
By Lemma \ref{lemma1} we get

	$$      \displaystyle\int_{\Omega} f d\mathbb{P} =  \displaystyle\lim_{n\to+\infty}\displaystyle\frac{1}{n} \displaystyle\sum_{j=1}^n \displaystyle\int_{\Omega} h_j d\mathbb{P}  = 0 .                                   $$
	Since $f \geq 0$ almost everywhere it follows that $f = 0$ almost everywhere. Therefore, $ \mathbb{P}(F) = 1$.           
\end{proof}

Another important set is defined by
$$      G = \{ \sigma \in \Omega:  \displaystyle\lim_{n\to+\infty}  Diam(\omega_{\sigma_n} \circ \cdot \cdot \cdot \circ \omega_{\sigma_1}(X))  = 0   \}.                                                             $$
It is obvious that $G \subset F$. If the IFS $\omega$ is weakly hyperbolic  by Lemma 2.2 of \cite{ASJ} it follows that $S=G=F = \Omega$. We will see in the last section an example of a $\mathbb{P}$-weakly hyperbolic IFS such that $\mathbb{P}(F) =  1$ but $\mathbb{P}(G) = 0$.

\begin{remark}
Let $\omega$ be an IFS on a compact interval with finite parameter space such that for every $i \in \Lambda$ the maps $\omega_i$ are continuous and one-to-one. Malicet in \cite{M} proved that if the maps $\omega_i$ do not have a fixed point common then $\mathbb{P}(G) = 1$ for any Bernoulli measure. In particular, from Lemma \ref{lemakey}  it follows that this IFS is $\mathbb{P}$-weakly hyperbolic.
\end{remark}

Given an IFS $\omega:\Lambda \times X \rightarrow X$ is natural to consider the skew product map $\Phi:\Omega \times X \rightarrow \Omega \times X$ defined by
$  \Phi(\sigma,x) = (\beta(\sigma),\omega_{\sigma_1}(x)).$  
 
 Let $\mu_1$ be a fixed point of the operator $T_p$. By Lemma 4.1 of \cite{ASJ} we have that the measure $\mathbb{P} \times \mu_1 $ is invariant by $\Phi$. Thus, to prove that the operator $T_p$ has a unique fixed point it suffices prove that the skew product $\Phi$ has a unique invariant measure $\xi$ such that $\pi_{*}(\xi) = \mathbb{P}$, where $\pi: \Omega \times X \rightarrow X$ is the projection defined by $\pi(\sigma,x) = \sigma$.
 
 \begin{proposition}\label{pro1}
 	Let $\mu_1$ and $\mu_2$ be invariant measures for $\Phi$. If $\nu(F) = 1$ and
 	$ \pi_{*}(\mu_1) = \pi_{*}(\mu_2)= \nu$ then $\mu_1 = \mu_2$.       
 \end{proposition}
 \begin{proof}
 	
 	Let $r$ be the distance in $\Omega \times X$, where $r((\sigma,x), (\xi,y)) = D(\sigma,\xi) + d(x,y)$. Consider a Lipschitz function $\varphi:\Omega \times X \rightarrow \mathbb{R}$ with Lipschitz constant $C$. By Birkhoff's ergodic theorem  there exists a Borel set $A$ with $   \mu_1(A) = 1$ such that if $(\sigma,x) \in A$ then
 	$$
 		\varphi^{*}(\sigma,x)= \displaystyle\lim_{n\to+\infty}  \displaystyle\frac{1}{n}\displaystyle\sum_{j=0}^{n-1} \varphi(\Phi^{j}(\sigma, x))
 	$$
 	exists.
 	
 	Consider the set $B = A \cap (F \times X)$. Since $\mu_1(F \times X) = \nu(F) = 1$ it follows that $\mu_1(B) = 1$. We claim that if $(\sigma,x) \in B$ then $\varphi^{*}(\sigma, y)$ exists, for all $y \in X$ and $ \varphi^{*}(\sigma, y) = \varphi^{*}(\sigma, x)         $. In fact, observe that
 	\begin{eqnarray*}
 		\Big|   \displaystyle\frac{1}{n}\displaystyle\sum_{j=0}^{n-1} \varphi(\Phi^{j}(\sigma, x)) - \displaystyle\frac{1}{n}\displaystyle\sum_{j=0}^{n-1} \varphi(\Phi^{j}(\sigma, y))                       \Big| & \leq& \displaystyle\frac{C}{n} \displaystyle\sum_{j=0}^{n-1} r(    \Phi^{j}(\sigma, x)        ,   \Phi^{j}(\sigma, y)       )\\
 		&\leq& \displaystyle\frac{C}{n} \displaystyle\sum_{j=1}^{n-1} Diam(\omega_{\sigma_{j}} \circ \cdot \cdot \cdot \circ \omega_{\sigma_1}(X))\\
 		&+& \displaystyle\frac{C}{n}Diam(X).\\
 	\end{eqnarray*}
 	Since $\sigma \in F$ by definition of $F$, 
 	$$       \displaystyle\lim_{n\to+\infty}    \displaystyle\frac{1}{n} \displaystyle\sum_{i=1}^n Diam(\omega_{\sigma_i} \circ \cdot \cdot \cdot \circ \omega_{\sigma_1}(X))  = 0.                             $$
 	Therefore, $\varphi^{*}(\sigma, y)$ exists and $ \varphi^{*}(\sigma, y) = \varphi^{*}(\sigma, x)         $.
 	
 	Consider the set
 	$$\Omega^{*} = \{\sigma \in F: {\rm there \ \ exists}\, \ \  x \in X  \ \ {\rm such \ \ that}\,  \ \ \varphi^{*}(\sigma, x) \ \ {\rm is}\, \ \ \rm{defined}\}.$$
 	
 	 We claim that $\nu(\Omega^{*}) =  \nu(F) =  1$. In fact, let us suppose that for some $D \subset F$, with $\nu(D) > 0$, if $\sigma \in D$ then  $\varphi^{*}(\sigma, x) $ do not exist for all $x \in X$. Observe that $\mu_1(D \times X) = \nu(D) > 0$, but this is an absurd because $\mu_1(B) = 1$.      
 	 
 	Consider a measurable function $g: \Omega \rightarrow \mathbb{R}$, where $g(\sigma) = \varphi^{*} (\sigma, x)$ for any $ x \in X$, if $\sigma \in \Omega^{*} $. Since $\mu_2( \Omega^{*} \times X ) = 1$ by  Birkhoff's ergodic theorem we have 	
 	\begin{eqnarray*} \displaystyle\int_{\Omega \times X} \varphi \,  d\mu_1 &=&   \displaystyle\int_{\Omega \times X} \varphi^{*} d\mu_1\\  
 		&=&      \displaystyle\int_{\Omega^{*} \times X} \varphi^{*} d\mu_1\\         
 		&=& \displaystyle\int_{\Omega} g \, d\nu\\
 		&=&  \displaystyle\int_{\Omega \times X} \varphi^{*} d\mu_2\\
 		&=&  \displaystyle\int_{\Omega \times X} \varphi \, d\mu_2.\\	                                                   
 	\end{eqnarray*}
 	Now, using that the subspace of the Lipschitz functions on $\Omega\times X$ is dense in the space of the continuous functions on $\Omega \times X$ it follows that 
 	$$         \displaystyle\int_{\Omega \times X} \psi \, d\mu_1 =      \displaystyle\int_{\Omega \times X} \psi \, d\mu_2,                           $$
 	for every continuous function $\psi$ on $\Omega \times X$ thus $\mu_1 = \mu_2$.
 \end{proof}
Since the IFS $\omega$ is $\mathbb{P}$-Weakly Hyperbolic, from Lemma \ref{lemakey} it follows that \linebreak $\mathbb{P}(F) = 1$. Thus by Proposition \ref{pro1} the skew product $\Phi$ has a unique invariant measure $\xi$ such that $\pi_{*}(\xi) = \mathbb{P}$. Therefore, the operator $T_p$ is asymptotic stable and has a unique fixed point $\mu_p$.

Now, we assume that $p(U) > 0$ for every open set $U \subset \Lambda$ to study the support of the measure $\mu_p$. Following \cite{ASJ} and \cite{MA} for each $\sigma \in \Omega$, $n \in \mathbb{N}$ and $x \in X$, define $\Gamma(\sigma,n,x) = \omega_{\sigma_1} \circ \cdot \cdot \cdot \circ \omega_{\sigma_n}(x)$. 	Observe that for each $n \geq 1$, we have
$$       \omega_{\sigma_1} \circ \cdot \cdot \cdot \circ \omega_{\sigma_{n+1}}(X) \subset \omega_{\sigma_1} \circ \cdot \cdot \cdot \circ \omega_{\sigma_n}(X).                        $$
If $\sigma \in S$ by definition we have that  $\displaystyle\lim_{n \rightarrow \infty} Diam(\omega_{\sigma_1} \circ \cdot \cdot \cdot \circ \omega_{\sigma_n}(X) ) = 0 $. Hence, the set
$$  \displaystyle\bigcap_{n=1}^{\infty}   \omega_{\sigma_1} \circ \cdot \cdot \cdot \circ \omega_{\sigma_n}(X)                                         $$
consists of a unique point $y$. On the other hand, 
$$d(y,\Gamma(\sigma,n,x ) ) \leq Diam(\omega_{\sigma_1} \circ \cdot \cdot \cdot \circ \omega_{\sigma_n}(X) ).$$
Therefore, $ \displaystyle\lim_{n\to+\infty}\Gamma(\sigma,n,x ) = y$ for any $x \in X$. This defines a function $\Gamma:S \rightarrow X$  given by 
$$     \Gamma(\sigma) =    \displaystyle\lim_{n\to+\infty}\Gamma(\sigma,n,x ),\ \ {\rm for}\, \, {\rm any}\, \,x \in X  .$$ 

By the same proof of Lemma 2.7 of \cite{ASJ} we have that the map $\Gamma$ is continuous.  We claim that ${\rm supp}\,(\mu) = \overline{\Gamma(S)}$. In fact, take $a = \Gamma(\sigma)$ with $ \sigma \in S $, observe that $\omega_{\lambda}(y) \in \Gamma(S)$ for any $\lambda \in \Lambda$ and $y\in \Gamma(S)$ thus for every $n$ we have
$$T_p^n(\delta_a)\big( \overline{\Gamma(S)} \big) = 1.$$

By the previous discussion we have that $T_p^n(\delta_a) \rightarrow \mu_p$ in the wea$k^{*}$ \linebreak topology. Therefore, $\mu_p(\overline{\Gamma(S)}) \geq \displaystyle\lim_{n\to+\infty} T_p^n(\delta_a)\big( \overline{\Gamma(S)} \big) = 1$, this proves that \linebreak $\rm{supp}(\mu_p) \subset \overline{\Gamma(S)}$.

Now take $p \in \overline{\Gamma(S)}$ and an open set $U \subset X$ with $p \in U$, by continuity of $\Gamma$ there exists $R > 0$, $m \in \mathbb{N}$ and open sets $U_1,...,U_m \subset \Lambda$ such that $\Gamma([1;U_1,...,U_m] \cap S) \subset \overline{B_R(p)} \subset U$. By definition of $T_p^n$ for $n > m$ we have
\begin{eqnarray*}
	T_p^n(\delta_a)(\overline{B_R(p)})
	&=& \displaystyle\int_{\Lambda^n} \delta_{\omega_{\lambda_1} \circ \cdot \cdot \cdot \circ \omega_{\lambda_n}(a)}(\overline{B_R(p)}) dp^n(\lambda)\\
	&=& \displaystyle\int_{\Omega} \delta_{\omega_{\lambda_1} \circ \cdot \cdot \cdot \circ  \omega_{\lambda_n}(a)}(\overline{B_R(p)}) d\mathbb{P}\\
\end{eqnarray*}
Observe that $\omega_{\lambda_1} \circ \cdot \cdot \cdot \circ \omega_{\lambda_n} (a) = \Gamma(\eta_{\lambda_1} \circ \cdot \cdot \cdot \circ \eta_{\lambda_n}(\sigma))$, where for each $\lambda \in \Lambda$ the map $\eta_{\lambda}: \Omega \rightarrow \Omega$ is defined by $\eta_{\lambda}(\xi_1,\xi_2,...) = (\eta,\xi_1,\xi_2,..)$. Hence, if $\lambda_i \in U_i$ for $i=1,...,n$ then $\omega_{\lambda_1} \circ \cdot \cdot \cdot \circ \omega_{\lambda_n} (a) = \Gamma(\eta_{\lambda_1} \circ \cdot \cdot \cdot \circ \eta_{\lambda_n}(\sigma)) \in \Gamma([1;U_1,...,U_m] \cap S)$. Since $\Gamma([1;U_1,...,U_m] \cap S) \subset \overline{B_R(p)} \subset U$ it follows that
\begin{eqnarray*}	
T_p^n(\delta_a)(\overline{B_R(p)})
&=& 	\displaystyle\int_{([1;U_1,...,U_m] \cap S)} \delta_{\omega_{\lambda_1} \circ \cdot \cdot \cdot \circ \omega_{\lambda_n}(a)}(\overline{B_R(p)}) d\mathbb{P}\\
	&\geq& \mathbb{P}([1;U_1,...,U_m]).\\
\end{eqnarray*}
Therefore, $\mu_p(U) \geq \mu_p(\overline{B_R(p)}) \geq \mathbb{P}([1;U_1,...,U_m]) >0$. This proves that \linebreak $\overline{\Gamma(S)} \subset  {\rm supp}\,(\mu)$ and so  $\overline{\Gamma(S)} = {\rm supp}\,(\mu)$. This concludes the proof of Theorem \ref{teoo}.
\begin{remark}
When the IFS $\omega$ is weakly hyperbolic the set $\Gamma(\Omega)$ coincides with the attractor of the IFS.
\end{remark}
\section{Proof of Theorem \ref{ergodic}}
Let $\nu$ be an invariant probability measure by $\beta$. Since $\Omega \times X$ is a compact metric space it follows that the space $\mathcal{M}_1(\Omega \times X)$ is a compact metric space. By continuity of $\Phi$ there exists at least one invariant measure $\nu^{\Phi}$ such that $\pi_{*}(  \nu^{\Phi} ) = \nu$. By Proposition \ref{pro1} if $\nu(F) = 1$ then  $\nu^{\Phi}$ is the unique invariant measure by $\Phi$ with this property. In the next result we relate the ergodicity of $\nu$ with the ergodicity of $\nu^{\Phi}$.

  \begin{proposition}\label{pro2}
	Let $\omega$ be  an IFS with $\nu(F) = 1$. If the measure $\nu$ is ergodic then the measure  $\nu^{\Phi}$ is ergodic.
\end{proposition}
\begin{proof}
	Let $f:\Omega \times X \rightarrow \mathbb{R}$ be a continuous function. Since the Lipschitz functions are dense in the space of the continuous functions on $\Omega \times X$ there exists a sequence of Lipschitz functions $f_n$ such that $||f_n -f||_{\infty} < 1/n$. By Birkhoff's ergodic theorem  there exists a Borel set $A$ with $  \nu^{\Phi}(A) = 1$ such that if $(\sigma,x) \in A$ then
	\begin{eqnarray*}
		f_n^{*}(\sigma,x)&=& \displaystyle\lim_{n\to+\infty}  \displaystyle\frac{1}{k}\displaystyle\sum_{j=0}^{k-1} f_n(\Phi^{j}(\sigma,x))
	\end{eqnarray*}
	exists for every $n \in \mathbb{N}$ and 
	\begin{eqnarray*}
		f^{*}(\sigma,x)&=& \displaystyle\lim_{k\to+\infty}  \displaystyle\frac{1}{k}\displaystyle\sum_{j=0}^{k-1} f(\Phi^{j}(\sigma,x))
	\end{eqnarray*}
	also exists. Consider the set $B = A \cap (F \times X)$. Since $\nu^{\Phi}(F \times X) = \nu(F) = 1$ it follows that $\nu^{\Phi}(B) = 1$. Using the same proof of Proposition \ref{pro1} for each $n \in \mathbb{N}$, if $(\sigma,x) \in B$ then $f_n^{*}(\sigma, y)$ exists, for all $y \in X$ and $ f_n^{*}(\sigma, y) = f_n^{*}(\sigma, x)         $.
	
	We claim that if $(\sigma,x) \in B$ then $f^{*}(\sigma, y)$ exists, for all $y \in X$ and $ f^{*}(\sigma, y) = f^{*}(\sigma, x)         $. In fact, let 
	$$     S_n(\sigma,x) = 	\Big|   \displaystyle\frac{1}{n}\displaystyle\sum_{j=0}^{n-1} f(\Phi^{j}(\sigma, x)) - \displaystyle\frac{1}{n}\displaystyle\sum_{j=0}^{n-1} f(\Phi^{j}(\sigma, y))                       \Big|.                              $$	
We have	
	\begin{eqnarray*}
		S_n(\sigma,x)
		&\leq&  \displaystyle\frac{1}{n} \Big|   \displaystyle\sum_{j=0}^{n-1} f(\Phi^{j}(\sigma, x)) - \displaystyle\sum_{j=0}^{n-1} f_k(\Phi^{j}(\sigma, x))                       \Big|                                             \\
		&+&  \Big| \displaystyle\frac{1}{n}  \displaystyle\sum_{j=0}^{n-1} f_k(\Phi^{j}(\sigma, x)) -\displaystyle\frac{1}{n} \displaystyle\sum_{j=0}^{n-1} f_k(\Phi^{j}(\sigma, y))                       \Big|                                             \\
		&+&\Big| \displaystyle\frac{1}{n}  \displaystyle\sum_{j=0}^{n-1} f_k(\Phi^{j}(\sigma, y)) -\displaystyle\frac{1}{n} \displaystyle\sum_{j=0}^{n-1} f(\Phi^{j}(\sigma, y))                       \Big|                                             \\
		&\leq&  \displaystyle\frac{2}{k} + \Big| \displaystyle\frac{1}{n}  \displaystyle\sum_{j=0}^{n-1} f_k(\Phi^{j}(\sigma, x)) -\displaystyle\frac{1}{n} \displaystyle\sum_{j=0}^{n-1} f_k(\Phi^{j}(\sigma, y))                       \Big|.  \\
	\end{eqnarray*}
	Given $\epsilon > 0$, take $k$ such that $1/k < \epsilon/2$. Since $ f_k^{*}(\sigma, y) = f_k^{*}(\sigma, x)         $  there exists $n_0$ such that if $n \geq n_0$ then 
	$$  \Big| \displaystyle\frac{1}{n}  \displaystyle\sum_{j=0}^{n-1} f_k(\Phi^{j}(\sigma, x)) -\displaystyle\frac{1}{n} \displaystyle\sum_{j=0}^{n-1} f_k(\Phi^{j}(\sigma, y))                       \Big| < \displaystyle\frac{\epsilon}{2}.                                                  $$
	Consequently, if $n \geq n_0$ then $$\Big|   \displaystyle\frac{1}{n}\displaystyle\sum_{j=0}^{n-1} f(\Phi^{j}(\sigma, x)) - \displaystyle\frac{1}{n}\displaystyle\sum_{j=0}^{n-1} f(\Phi^{j}(\sigma, y))                       \Big|  \leq \epsilon$$
	
	Therefore, $ f^{*}(\sigma, y) = f^{*}(\sigma, x)         $. Now consider the set  
	$$\Omega^{*} = \{\sigma \in F: {\rm there \ \ exists}\, \ \  x \in X  \ \ {\rm such \ \ that}\,  \ \ f^{*}(\sigma, x) \ \ {\rm is}\, \ \ \rm{defined}\}.$$

	We claim that $\nu(\Omega^{*}) =  \nu(F) =  1$. In fact, let us suppose that for some $D \subset F$, with $\nu(D) > 0$, if $\sigma \in D$ then  $f^{*}(\sigma, x) $ do not exist for all $x \in X$. Observe that $\nu^{\Phi}(D \times X) = \nu(D) > 0$, but this is an absurd because $\nu^{\Phi}(B) = 1$.

	
	Now consider a measurable function $g: \Omega \rightarrow \mathbb{R}$, where $g(\sigma) = f^{*} (\sigma, x)$ for any $ x \in X$, if $\sigma \in \Omega^{*} $. Observe that $g \circ \beta(\sigma) = g(\sigma)$ if  $\sigma \in \Omega^{*}$, from ergodicity of $(\beta, \nu)$ it follows that $g$ is constant for $\nu$-a.e $\sigma \in \Omega$. Hence $f^{*}$ is constant for   $\nu^{\Phi}$-a.e $(\sigma,x) \in \Omega \times X$. Since $ \Omega \times X$ is a compact metric space it follows that $\nu^{\Phi}$ is ergodic. 
	\end{proof}
Since the IFS $\omega$ is  $\mathbb{P}$-weakly hyperbolic by Lemma \ref{lemakey} it follows that $\mathbb{P}(F) = 1$. From ergodicity of $(\mathbb{P},\beta)$ and from Proposition \ref{pro2} it follows that $\mathbb{P}^{\Phi} = \mathbb{P} \times \mu_p$ is a ergodic measure with respect to $\Phi$.  Given a continuous function $f:X \rightarrow \mathbb{R}$, consider the continuous function $\varphi:\Omega \times X \rightarrow \mathbb{R}$ where $\varphi(\sigma,x) = f(x)$. By Birkhoff's ergodic
theorem for $(\mathbb{P}\times \mu_p)$-a.e. $(\sigma,x) \in \Omega \times X$ we have
$$    \displaystyle\lim_{n\to+\infty}  \displaystyle\frac{1}{n}\displaystyle\sum_{j=0}^{n-1} \varphi(\Phi^{j}(\sigma, x)) =    \displaystyle\lim_{n\to+\infty}    \displaystyle\frac{1}{n} \displaystyle\sum_{j=1}^{n} f(  \omega_{\sigma_j} \circ \cdot \cdot \cdot \circ \omega_{\sigma_1}(x)   ) = \displaystyle\int_{X} f \, d\mu.                                                            $$
On the other hand, by the proof of  Proposition  \ref{pro2} there exists a subset $\Omega^{*} \subset F$ with $\mathbb{P}(\Omega^{*} ) = 1$ such that if $\sigma \in \Omega^{*}$ then for any $x,y \in X$ we have
$$   \displaystyle\lim_{n\to+\infty}  \displaystyle\frac{1}{n}\displaystyle\sum_{j=0}^{n-1} \varphi(\Phi^{j}(\sigma, x))          = \displaystyle\lim_{n\to+\infty}  \displaystyle\frac{1}{n}\displaystyle\sum_{j=0}^{n-1} \varphi(\Phi^{j}(\sigma, y)) = \displaystyle\int_{X} f \, d\mu.                                                              $$ 
This concludes the proof of Ergodic theorem for $\mathbb{P}$-weakly hyperbolic IFS.
\section{Examples}
When the IFS is Weakly Hyperbolic by Lemma 2.2 of \cite{ASJ} it follows that \linebreak $S=G= \Omega$, however if the IFS is only $\mathbb{P}$-Weakly Hyperbolic it is possible that $\mathbb{P}(G) = 0$. We shall see in the example below an IFS where $\mathbb{P}(S) = 1$ but $\mathbb{P}(G) =0$.

The idea to do this example comes from discussions made by \"Oberg in \cite{O}. Consider the IFS where $\Lambda = \{0,1\}$, $X=[0,1]$, $\omega_0(x)=x/2$ and $\omega_1(x) = 2x$ for $0 \leq x \leq 1/2$, $\omega_1(x) = 1$ for $1/2 \leq x \leq 1$. 
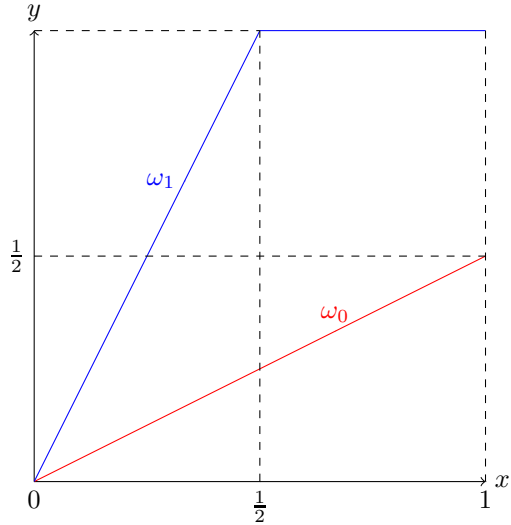
\begin{center}
	\begin{tikzpicture}
	\draw[->] (0,0) -- (6,0) node[right] {$x$};
	\draw[->] (0,0) -- (0,6) node[above] {$y$};
	\draw[color=red] (0,0) -- (6,3);
	\draw[color=blue] (0,0) -- (3,6);
	\draw[color=blue] (3,6) -- (6,6);
	\draw[dashed] (0,6) -- (3,6);
	\draw[dashed] (6,6) -- (6,0);
	\draw[dashed] (3,0) -- (3,6);
	\draw[dashed] (0,3) -- (6,3);
	\fill[black] (3,0) circle (0.mm) node[below]
	 {$\frac{1}{2}$};
	 \fill[black] (0,0) circle (0.mm) node[below] {$0$};
	 \fill[black] (6,0) circle (0.mm) node[below] {$1$};
	 \fill[black] (0,3) circle (0.mm) node[left] {$\frac{1}{2}$};
	 \fill[red] (4,2) circle (0.mm) node[above] {$\omega_0$};
	 \fill[blue] (2,4) circle (0.mm) node[left] {$\omega_1$};
	\end{tikzpicture}
\end{center}
Let $p$ be the probability in $\Lambda$, where $p(\{0\}) = p(\{1\}) = 1/2$. We claim that $\mathbb{P}(G) = 0$. In fact, it is know that there exists a set $A \subset \Omega$ with $\mathbb{P}(A) =1$ such that if $\sigma \in A$ then there exists a sequence $n_k \rightarrow \infty$ such that $\displaystyle\sum_{j=1}^{n_k} (-1)^{\sigma_j} = 0$, see \cite{Bii} for example. Take $\sigma \in A$, if there exists a sequence $m_k \rightarrow \infty$ such that $Diam(\omega_{m_k} \circ \cdot \cdot \cdot \circ \omega_{1}([0,1])) \geq 1/2$ then $\sigma \notin G$. Suppose now that there exists $N_0$ such that if $n \geq N_0$ then $Diam(\omega_{n} \circ \cdot \cdot \cdot \circ \omega_{1}([0,1])) \leq 1/2.$

Observe that $0 \in \omega_{n} \circ \cdot \cdot \cdot \omega_{1}([0,1])$ for every $n$. Thus for $n \geq N_0$ we have that $\omega_{n} \circ \cdot \cdot \cdot \omega_{1}([0,1]) \subset [0,1/2]$. On the other hand,
if $x \in[0,1/2]$ then 
$$\omega_0 \circ \omega_1 (x) =\omega_1 \circ \omega_0 (x) = x .$$
Therefore for $n > N_0$ we have
$$  Diam(\omega_{n} \circ \cdot \cdot \cdot \circ \omega_{1}([0,1])) = 2^{r_1}2^{-r_0}  Diam(\omega_{N_0} \circ \cdot \cdot \cdot \circ \omega_{1}([0,1])),                                                    $$
where $r_1 = \# \{N_0 < j \leq n: \sigma_j = 1\}$ and $r_0 = \# \{N_0 < j \leq n: \sigma_j = 0\}$. Since $\sigma \in A$ there exists there exists a sequence $n_k \rightarrow \infty$ such that $\displaystyle\sum_{j=1}^{n_k} (-1)^{\sigma_j} = 0$ so for $n_k > N_0$ we have that $   Diam(\omega_{n_k} \circ \cdot \cdot \cdot \circ \omega_{1}([0,1])) \geq 2^{-N_0}   Diam(\omega_{N_0} \circ \cdot \cdot \cdot \circ \omega_{1}([0,1])).$
Therefore, $\sigma \notin G$ and so $\mathbb{P}(G) = 0$.

 Now we will show that $\mathbb{P}(S) =1$. It is know that for each $l \in \mathbb{Z}$ there exists a set $A_l$ with $\mathbb{P}(A_l) = 1$ such that if $\sigma \in A$ then there exists a sequence $n_k \rightarrow \infty$ such that $\displaystyle\sum_{j=1}^{n_k} (-1)^{\sigma_j} = l$, see \cite{Bii} for example. Take $\sigma \in \bigcap_{l=1}^{\infty} A_l$, for each $l >0$ there exists $n_l$ such that $\displaystyle\sum_{j=1}^{n_k} (-1)^{\sigma_j} = l$. Hence,
$  Diam(\omega_{1} \circ \cdot \cdot \cdot \circ \omega_{n_l}([0,1]))  \leq 2^{-l}.                     $
On the other hand, since $\displaystyle\lim_{n\to+\infty}  Diam(\omega_{1} \circ \cdot \cdot \cdot \circ \omega_{n}([0,1]))$ exists it follows that $\displaystyle\lim_{n\to+\infty}  Diam(\omega_{1} \circ \cdot \cdot \cdot \circ \omega_{n}([0,1])) = 0$. This proves that $\mathbb{P}(S) =1$ and by Lemma \ref{lemakey} also we have that $\mathbb{P}(F) =1$.

\end{document}